\documentclass[12pt]{elsarticle}

\usepackage{amssymb,amsmath,bbm,amsthm}
\newtheorem{theorem}{Theorem}[section]
\newtheorem{proposition}{Proposition}[section]
\newtheorem{remark}{Remark}[section]
\newtheorem{lemma}{Lemma}[section]
\newtheorem{corollary}{Corollary}[section]
\newtheorem{definition}{Definition}[section]

\def\SU{{\bf SU}}
\def\SL{{\bf SL}}

\def\a{\mathfrak{a}}
\def\un{\mathbbm{1}}

\def\F{{\bf F}}
\def\R{{\bf R}}
\def\C{{\bf C}}
\def\H{{\bf H}}

\def\dual{\mathop{\mathcal{A}^*}}
\DeclareMathOperator{\tr}{tr}
\DeclareMathOperator{\diag}{diag}

\begin{document}

\begin{frontmatter}

\title{A Laplace-type representation of the generalized spherical functions associated to the root systems of type $A$}


\author{Patrice Sawyer}

\address{Department of Mathematics \& Computer Science, Laurentian University, Sudbury, Ontario, Canada K1K 4W5}

\begin{abstract}
In this paper, we extend the iterative expression for the generalized spherical functions associated to the root systems of type $A$ previously obtained beyond regular elements.  We also provide 
the corresponding expression in the flat case.  From there, we derive a Laplace-type representation for the generalized spherical functions associated to the root systems of type $A$ in the Dunkl setting as well as in the trigonometric Dunkl setting.  This representation leads us to describe precisely the support of the generalized Abel transform.  Thanks to a recent result of Rejeb, this allows us to give the support for the Dunkl intertwining operator.

\end{abstract}

\begin{keyword}
generalized spherical function, Dunkl, root system, intertwining operator, Abel transform, dual of the Abel transform
\end{keyword}

\end{frontmatter}

This research is supported by funding from Laurentian University.

The author is thankful to the Institut f\"ur Mathematik at the Universi\"at Paderborn for their hospitality in July 2013 during which this work was started and to Professor Margit R\"osler for helpful conversations.

\section{Introduction}\label{intro}
We start by providing some background.  We refer the reader to \cite{Narayanana,Opdam1,Roesler1} for a more complete exposition on the Dunkl and trigonometric Dunkl settings.  Given a root system $R$ and a Cartan subalgebra $\a$, for every root
$\alpha\in R$, let $r_\alpha(X)=X-2\,\frac{\langle \alpha,X\rangle}{\langle \alpha,\alpha\rangle}\,\alpha$ and let $\partial_\xi$ be the derivative in the direction of $\xi$.  
The Dunkl operators indexed by $\xi\in \a$ are then given by
\begin{align*}
T_\xi&=\partial_\xi+\sum_{\alpha\in R_+}\,k_\alpha\,\alpha(\xi)\,\frac{1}{\langle \alpha,X\rangle}\,(1-r_\alpha)
\end{align*}
The Weyl group $W$ associated to the root system is generated by the reflection maps $r_\alpha$.

In a similar manner, the trigonometric Dunkl operators (also called Dunkl-Cherednik operators or simply Cherednik operators) are given by
\begin{align*}
D_\xi&=\partial_\xi+\sum_{\alpha\in R_+}\,k_\alpha\,\alpha(\xi)\,\frac{1}{1-e^{-\alpha}}\,(1-r_\alpha)-\rho(k)(\xi).
\end{align*}

In the Dunkl setting, the function $G(\lambda,\cdot)$ is defined as the unique analytic solution of 
\begin{align}
T_\xi\,G(\lambda,\cdot) = \langle\xi,\lambda\rangle\,G(\lambda,\cdot)\label{T}
\end{align}
with $G(\lambda,0)=1$.

We can also define the generalized spherical functions as follows: 
$F(\lambda,\cdot)$ is the unique analytic function such that for every symmetric polynomial $p$ (\emph{i.e.} a polynomial which is invariant with respect to the action of the Weyl group), we have
\begin{align}
p(T_\xi)\,F(\lambda,\cdot) = p(\langle\xi,\lambda\rangle)\,F(\lambda,\cdot)\label{F}
\end{align}
with $F(\lambda,0)=1$.  Note that 
\begin{align}
F(\lambda,X)&=\frac{1}{|W|}\,\sum_{w\in W}\,G(\lambda,w\cdot X).\label{FG}
\end{align}

The definitions in (\ref{T}), (\ref{F}) and (\ref{FG}) are essentially the same in the trigonometric setting except that we then replace $T_\xi$ by $D_\xi$ (refer for example to
\cite{Narayanana,Opdam1}).

With some adjustment in the spectral parameter $\lambda$, the functions $F(\lambda,\cdot)$ generalize the spherical functions on symmetric spaces of Euclidean type (in the Dunkl setting) and those on the symmetric spaces of noncompact type (in the trigonometric Dunkl setting).  We refer the reader to Helgason's books \cite{Helgason1, Helgason2} as the standard reference on symmetric spaces.

Indeed, for selected choice of root multiplicities, the root system corresponds to a symmetric space of noncompact type in the trigonometric Dunkl setting and to the corresponding 
flat symmetric space in the Dunkl setting.  We then say that we are in the ``group case'' or ``in the geometric setting''. For example, in the case of the trigonometric setting for the root system of type $A_{n-1}$, $m=1$, 2 or $4$ corresponds to the spaces $\SL(n,\F)/\SU(n,\F)$ with $\F=\R$, $\C$ or $\H$ and $m=\dim_{\R}\,\F$ (in addition, when $n=3$, $m=8$ also give the space $\SL(3,{\bf O})/\SU(3,{\bf O})\simeq \mathbf{E}_{6(-26)}/\mathbf{F}_4$ where ${\bf O}$ denotes the octonions).   In the Dunkl setting, for the same choice of the multiplicity $m$, we have the corresponding flat symmetric spaces.

It is well-know that in the geometric setting, if $H\not=0$ then the spherical functions have a Laplace type representation 
\begin{align}
\phi_\lambda(e^X)=\int_{\a}\,e^{i\,\langle\lambda,H\rangle}\,K(H,X)\,dH
\label{Laplace}
\end{align}
with $K(H,X)>0$ and where the support of $K(\cdot,X)$ is $C(X)$ ($H\not=0$ ensures that $\dim C(X)=\hbox{rank}$ by \cite[Theorem 10.1, Chap.{} IV]{Helgason2}).  

The function $K(H,\cdot)$ is the kernel of the Abel transform
\begin{align*}
\mathcal{A}(f)(H)=\int_{\a}\,f(e^X)\,K(H,X)\,\delta(X)\,dX
\end{align*}
while the dual Abel transform is simply given by
\begin{align*}
\mathcal{A}^*(f)(X)=\int_{\a}\,f(e^H)\,K(H,X)\,dH
\end{align*}
so that $\phi_\lambda=\mathcal{A}^*(e^{i\,\langle\lambda,\cdot\rangle})$.

As for the function $G(\lambda,\cdot)$ in the Dunkl setting, we have the following representation
\begin{align*}
G(\lambda,\cdot)&=V\,e^{\langle\lambda,\cdot\rangle}
\end{align*}
where $V$ is called the Dunkl's intertwining operator.  It is defined by the following properties
\begin{align*}
\left\lbrace
\begin{array}{cc}
T_\xi\,V&=V\,\partial_\xi,\\
V(\mathcal{P}_n)&=\mathcal{P}_n,\\
\left.V\right|_{\mathcal{P}_0}&=\hbox{id}
\end{array}
\right.
\end{align*}
where $\mathcal{P}_n$ is the space of homogeneous polynomials of degree $n$.  We also introduce the positive measure $\mu_x$ such that
\begin{align*}
Vf(X)=\int_{\a}\,f(H)\,d\mu_X(H)
\end{align*}
(for the existence of the positive measure, see for example \cite{Roesler2}).

Compare with the intertwining properties of the generalized Abel transform and its dual
\begin{align*}
p(\partial_\xi)\circ\mathcal{A}=\mathcal{A}\circ p(T_\xi),\\
p(T_\xi)\circ\mathcal{A}^*=\mathcal{A}^*\circ p(\partial_\xi)\\
\end{align*}
for every symmetric polynomial $p$.

From now on, unless otherwise mentioned, we are only concerned with the root systems of type $A$.  The superscript $(m)$ on the various objects  will serve to indicate that the associated multiplicity is equal to $m$ (\emph{e.g.} $\phi^{(m)}_\lambda$, $K^{(m)}(H,X)$, etc.).

In Section \ref{generalized}, we recall the recursive formulae (equation (\ref{Spherical0}) and (\ref{psi})) for the generalized spherical functions $\phi^{(m)}_\lambda$ (denoted $F(\lambda,\cdot)$ above) 
associated with the root system $A_{n-1}$ with root multiplicity $\Re m>0$ in the trigonometric setting.  These formulae were derived for $X\in\a^+$ (see for instance \cite{Sawyer1,Sawyer2}).  
We show first that (\ref{psi}) makes sense for all $X\in\a$.  We then derive Theorem \ref{walls} and Theorem \ref{walls2} to extend (\ref{Spherical0}) and (\ref{psi}) to the cases where $X\in \a$ is not regular. These two results are interesting in themselves. 

In Section \ref{Rep}, we show that in the case of the root system of type $A_{n-1}$ with arbitrary multiplicity $\Re m>0$, equation (\ref{Laplace}) still holds with the kernel
$K^{(m)}(\cdot,X)$ supported in the set $C(X)$ and that when $m>0$, $K^{(m)}(\cdot,X)>0$ on $C(X)^\circ$.  

In Section \ref{Dunkl}, with the help of a theorem of de Jeu \cite{DeJeu}, we extend the results of Section \ref{Rep} to the Dunkl setting.  We also use a result by Rejeb to deduce the exact support of the intertwining transform $V$.

\section{The generalized spherical function associated to the root systems of type $A$}\label{generalized}

We recall some preliminary definitions and results from \cite{Sawyer1, Sawyer2}.  In particular, we describe here the family of differential operators which are instrumental in defining the generalized spherical functions related to the root system $A_{n-1}$.

In what follows, $\a$ is the space of real $n\times n$ diagonal matrices and $\a^+$ is the subset with strictly decreasing diagonal entries.  
For simplicity, we will not 
assume here that the matrices have trace equal to 0 (refer however to Remark \ref{trace}).  
We will use lowercase to write the diagonal entries of an element of $\a$ (e.g. if $X\in \a$ then $X=
\diag[x_1,\dots,x_n]$).  We describe the action of the Weyl group on the elements of $\a$ as follows: if $\sigma\in W=S_n$ then $\sigma\cdot X=\diag[x_{\sigma^{-1}(1)},\dots,x_{\sigma^{-1}(n)}]$.

The differential operators $D^{(m)}_1$, \dots, $D^{(m)}_n$ defined below generate the algebra of differential operators $p(D_\xi)$, where $p$ is any symmetric
polynomial. 

\begin{definition}\label{DIFF}
Let $Y$ be an indeterminate, $\delta=(n-1,n-2,\dots,1,0)$ and let
\begin{align*}
D_n(Y;m) &= \prod_{p<q}\,(e^{2\,x_p}-e^{2\,x_q})^{-1}\,\sum_{\sigma\in S_n}\,\epsilon(\sigma)
\,e^{2\,\sum_{k=1}^n\,\delta(\sigma(k))\,x_k}
\\&\qquad\cdot
\prod_{s=1}^n
\,(Y+\delta(\sigma(s))+\frac{1}{m}\,\frac{\partial~}{\partial x_s})\\
&=\sum_{r=0}^n\,D_r^{(m)}\,Y^{n-r}
\end{align*}
(refer to \cite[page 190]{Macdonald} for more details).  Note in particular that $D_0^{(m)}=1$,
\begin{align*}
D_1^{(m)}
&= \frac{1}{m}\,\overbrace{\sum_{k=1}^n\,\frac{\partial~}{\partial x_k}}^{L_1}+\frac{n\,(n-1)}{2},\\
D_2^{(m)}
&= 
-\frac{1}{2\,m^2}\,\overbrace{\left[\sum_{i=1}^n\,\frac{\partial^2~~}{\partial   x_i^2}
+m\,\sum_{i<j}\,\coth(x_i-x_j)\,\left(\frac{\partial ~}{\partial x_i}-\frac{\partial ~}{\partial x_j}\right)
\right]}^{L_2^{(m)}}\\&
+\frac{1}{2\,m^2}\,L_1^2+\frac{(n-1)^2}{2\,m}\,L_1
+\frac {n\,(n-1)\,(n-2)\,(3\,n-1)}{24}.
\end{align*}
\end{definition}

\begin{definition}\label{defphi}
The generalized spherical function $\phi^{(m)}_\lambda$ for the root system $A_{n-1}$ is the unique analytic solution of the 
system
\begin{align*}
D_n(Y;m)\,\phi^{(m)}_\lambda(e^X)
&=\prod_{k=1}^n\,(Y+(n-1)/2+i\,\lambda_k/m)\,\phi^{(m)}_\lambda(e^X)
\end{align*}
with $\phi^{(m)}_\lambda(e^0)=1$.
\end{definition}

In \cite{Sawyer1, Sawyer2}, we proved the following result for the generalized spherical functions associated to the root system $A_{n-1}$.

\begin{theorem}\label{old}
For $X\in \a^+$, we define $\phi_\lambda^{(m)}(e^X)=e^{i\,\lambda(X)}$ when $n=1$ and, for 
$n\geq 2$,
\begin{align}
\phi^{(m)}_\lambda(e^X)
=\frac{\Gamma(m\,n/2)}{(\Gamma(m/2))^n}
e^{i\,\lambda_n\,\sum_{k=1}^n\,x_k}
\int_{E(X)}\, \phi^{(m)}_{\lambda_0}(e^\xi)\,S^{(m)}(\xi,X)\,d(\xi)^m\,d\xi
\label{Spherical0}
\end{align}
where $E(X)
=\{\xi=(\xi_1,\dots,\xi_{n-1})\colon x_{k+1}\leq \xi_k\leq x_k\}$, $\lambda(X)=\sum_{j=1}^n\,\lambda_j\,x_j$,
$\lambda_0(\xi)=\sum_{i=1}^{n-1}\,(\lambda_i-\lambda_n)\,\xi_i$, $d(X)=\prod_{r<s}\,\sinh(x_r-x_s)$, $d(\xi)=\prod_{r<s}\,\sinh(\xi_r-\xi_s)$ and
\begin{align*}
\lefteqn{S^{(m)}(\xi,X)}\\
&=d(X)^{1-m}\,d(\xi)^{1-m}\,\left[\prod_{r=1}^{n-1}
\,\left(\prod_{s=1}^r\,\sinh(x_s-\xi_r)
\,\prod_{s=r+1}^n\,\sinh(\xi_r-x_s)\right)\right]^{m/2-1}.
\end{align*}

Furthermore, if $\chi_\lambda=\phi_{\lambda-i\,\rho}$ then 
\begin{align}
\chi_\lambda^{(m)}(e^X)
&=\frac{\Gamma(m\,n/2)}{(\Gamma(m/2))^n}
\,e^{i\,\lambda_n\,\sum_{k=1}^n\,x_k}
\,\int_{\sigma_n}\,\chi^{(m)}_{\lambda_0}(e^\xi)
\,(\beta_1\,\cdots\,\beta_n)^{m/2-1}\,d\beta\label{psi}
\end{align}
where $\sigma_n=\{(\beta_1,\dots,\beta_n)\in\R^n\colon~\sum_{k=1}^n\,\beta_k=1\}$ and $\xi$ is defined 
by the relations
\begin{align}
\beta_p&=\frac{\prod_{j=1}^{n-1}\,(e^{2\,\xi_j}-e^{2\,x_p})}{\prod_{j\not=p}\,(e^{2\,x_j}-e^{2\,x_p})},~p=1,\dots,n.
\label{betap}
\end{align}

Then, for arbitrary $m$ with $\Re m>0$, $\phi^{(m)}_\lambda(e^X)$ are the generalized spherical functions associated to the root system $A_{n-1}$ 
as described in Definition \ref{defphi}.
\end{theorem}

We will assume $\Re m>0$ in this paper unless otherwise specified.  The next remark explains why normalizing our matrices to trace equal to 0 is not a concern.

\begin{remark}\label{trace}
It is readily seen that $\phi_{\lambda+a\,\scriptstyle\tr}(e^H)=e^{i\,a\,\tr H}\,\phi_\lambda(e^H)$ for $a\in \C$ and 
$\phi_{\lambda}(e^{H+b\,I})=e^{i\,b\,\sum_{k=1}^n\,\lambda_k}\,\phi_\lambda(e^H)$ for $b\in\R$ using induction.
\end{remark}

Our next step is to extend Theorem \ref{old} to an arbitrary $X\in \a$ using expression (\ref{psi}).  We start with a definition and notation.

\begin{definition}\label{notaplus}
For $X\in\a$, we will write $\pi(X)$ for the unique element in $(W\cdot X)\cap\overline{\a^+}$ (the projection of $X$ 
into $\overline{\a^+}$). We will also write
\begin{align}
\pi(X)=\diag\,[a_1,\dots,a_1,a_2,\dots,a_2,\dots,a_r,\dots,a_r]\label{a}
\end{align}
where $\pi(X)=\sigma\cdot X\in\overline{\a^+}$ for some $\sigma\in W$ 
and where the $a_i$'s are distinct, decreasing and the size of a given block of $a_i$'s is $n_i$.  We will also use 
the notation $N_0=0$, $N_k=n_1+\dots+n_k$ when $1\leq k\leq r$ (observe that $N_r=n$).
\end{definition}

We then need the following auxiliary result.
\begin{lemma}\label{TFSAE}
~
\begin{enumerate}
\item If $X\in \a^+$ then (\ref{betap}) is equivalent to $e^{2\,\xi_i}$, $i=1$, \dots, $n-1$, being the roots of 
\begin{align}
q(x)&=\sum_{p=1}^n\,\beta_p\,\prod_{i\not=p}\,(x-e^{2\,x_i}),~
\hbox{$x$ an indeterminate}\label{poly}
\end{align}
\emph{i.e.} $q(x)=\prod_{i=1}^{n-1}\,(x-e^{2\,\xi_i})$.

\item Let $X\in\a$.  If $\beta\in \sigma_n$ then the roots $u_1\geq u_2\geq \dots\geq u_{n-1}$ of $q(x)$ in 
(\ref{poly}) are strictly positive and if we write $\xi_i=(\log u_i)/2$, $i=1$, \dots, $n-1$, 
then $\xi\in E(\pi(X))$.

\item The map $\chi_\lambda$ can be extended continuously to a Weyl-invariant map over $\a$ by using equation 
(\ref{psi}) where $u_i=e^{2\,\xi_i}$, $i=1$, \dots, $n-1$, are the roots of $q(x)$ in (\ref{poly}).

\end{enumerate}

\end{lemma}

\begin{proof}
~
\begin{enumerate}
\item Write $e_k(\mu)=\sum_{i_1<\dots<i_k<n}\,\mu_{i_1}\cdots\,\mu_{i_k}$.
As pointed out in \cite{Sawyer1}, the relations in (\ref{betap}) are equivalent to 
\begin{align*}
e_k(e^{2\,\xi}) &=\sum_{p=1}^n\,\beta_p\,e_k
(e^{2\,x_1},\dots,e^{2\,x_{p-1}},e^{2\,x_{p+1}},\dots,e^{2\,x_n}),~k=0,\dots,n-1,
\end{align*}
which are in turn equivalent to 
$\prod_{i=1}^{n-1}\,(x-e^{2\,\xi_i})=q(x)$
(it suffices to consider the coefficients of $x^k$, $k=1$, \dots, $n-1$ on both sides).

\item Now suppose that 
$(\beta_1,\dots,\beta_n)\in\sigma_n$ is given and consider the polynomial (\ref{poly}) and the notation in 
Definition \ref{notaplus}.  Let $\gamma_k=\sum_{x_p=a_k}\,\beta_p$ and 
note that $S=\{s\colon \gamma_s>0\}\not=\emptyset$ since $\sum_{k=1}^r\,\gamma_k=1$. We have
\begin{align*}
q(x)&=\prod_{j=1}^r\,(x-e^{2\,a_j})^{n_j-1}\,\sum_{k=1}^r\,\gamma_k\,\prod_{j\not=k}\,(x-e^{2\,a_j})\\
&=\prod_{j=1}^r\,(x-e^{2\,a_j})^{n_j-1}
\prod_{s\not\in S}\,(x-e^{2\,a_s})
\,\sum_{s\in S}\,\gamma_s\,\overbrace{\prod_{j\not=s,j\in S}\,(x-e^{2\,a_j})}^{q_0(x)}.
\end{align*}

For $s\in S$, $q_0(e^{2\,a_s})$ alternate signs: there are $|S|-1$ roots.  These roots complement the 
roots $x=e^{2\,a_j}$ with multiplicities $n_j-1$ and the roots $x=e^{2\,a_j}$ with $j\not\in S$.

\item It suffices to reflect that in expression (\ref{psi}), the variable $\xi$ depends continuously on the coefficients $\beta_p$ (up to their order).  Using induction, the rest follows.

\end{enumerate}
\end{proof}

We are looking for simplified expressions corresponding to (\ref{Spherical0}) and (\ref{psi}) which 
are also valid for $X\in \overline{\a^+}$.  The following lemma is the basis for the computations required to prove Theorem \ref{walls}.

\begin{lemma}\label{beta}
We have
\begin{align*}
\lefteqn{\int_{\sigma_n}\,f(\beta_1+\dots+\beta_k,\beta_{k+1},\dots,\beta_n)\,(\beta_1\cdots \beta_n)^{m/2-1}\,d\beta}\\
&=\frac{\Gamma(m/2)^k}{\Gamma(m\,k/2)}\,\int_{\sigma_{n+1-k}}\,f(\gamma_k,\dots,\gamma_n)\,
\gamma_k^{m\,k/2-1}\,(\gamma_{k+1}\cdots\gamma_n)^{m/2-1}
\,d\gamma.
\end{align*}
\end{lemma}

\begin{proof}
It suffices to use the change of variables 
$\gamma_k=\beta_1+\dots+\beta_k$, $\tilde{\beta}_i=\beta_i/\gamma_k$, $i=1$, \dots, $k$ and 
$\gamma_j=\beta_j$, $j=k+1$, \dots, $n$ and to ``integrate out'' $\tilde{\beta}_1$, \dots, $\tilde{\beta}_k$
noting that $\sum_{i=1}^k\,\tilde{\beta}_i=1$ and $\sum_{j=k}^n\,\gamma_k=1$.
\end{proof}

This brings us to the following result which fully extends (\ref{psi}).

\begin{theorem}\label{walls}
Let $X$ and $n_1$, \dots, $n_r$ be as in (\ref{a}) and let $N_k=n_1+\dots+n_k$ ($N_0=0$).  Then
\begin{align}\label{psi2}
\chi_\lambda^{(m)}(e^X)
&=\frac{\Gamma(m\,n/2)}{\prod_{i=1}^r\,\Gamma(m\,n_i/2)}
\,e^{i\,\lambda_n\,\sum_{k=1}^r\,n_k\,a_k}
\,\int_{\sigma_r}\,\chi^{(m)}_{\lambda_0}(e^{\xi})
\,\prod_{i=1}^r\gamma_i^{m\,n_i/2-1}\,d\gamma
\end{align}
where the $e^{2\,\xi_i}$ are the roots of $q(x)=\prod_{j=1}^r\,(x-e^{2\,a_j})^{n_j-1}\,\sum_{i=1}^r
\,\gamma_i\,\prod_{j\not=i}\,(x-e^{2\,a_j})$.

More precisely,
\begin{align}
\xi&=\diag[\overbrace{a_1,\dots,a_1}^{n_1-1},\eta_1,\overbrace{a_2,\dots,a_2}^{n_2-1},\eta_2,\dots,\label{aa}\\
&\qquad\qquad\qquad \nonumber
\overbrace{a_{r-1},\dots,a_{r-1}}^{n_1-1},\eta_{r-1},\overbrace{a_r,\dots,a_r}^{n_r-1}]\in E(X)
\end{align}
with $\xi_k=a_i$ whenever $N_i<i<N_{i+1}$ and $\xi_{N_i}=\eta_i\in [a_i,a_{i+1}]$ are the roots of the 
polynomial $q_1(x)=\sum_{i=1}^r\,\gamma_i\,\prod_{j\not=i}\,(x-e^{2\,a_j})$.  Equivalently,
\begin{align*}
\gamma_p&=\frac{\prod_{j=1}^{r-1}\,(e^{2\,\eta_j}-e^{2\,a_p})}{\prod_{j\not=p}\,(e^{2\,a_j}-e^{2\,a_p})},
~p=1,\dots,r.
\end{align*}
\end{theorem}

\begin{proof}
This follows from equation (\ref{psi}) and repeated use of Lemma \ref{beta}.
\end{proof}

\begin{theorem}\label{walls2}
We use the notation of Theorem \ref{walls} and assume that $X\in\overline{\a^+}$ for simplicity.  Then 
\begin{align}
\phi_\lambda^{(m)}(e^X)
&=\frac{\Gamma(m\,n/2)}{\prod_{i=1}^r\,\Gamma(m\,n_i/2)}
\,e^{i\,\lambda_n\,\sum_{k=1}^r\,n_k\,a_k}
\,\int_{E(X)}\,\phi^{(m)}_{\lambda_0}(e^\xi)\,\tilde{S}^{(m)}(\eta,X)
\label{phispecial}\\\nonumber &\qquad\qquad\qquad\cdot
\,\prod_{i<j}\,\sinh^m(\eta_i-\eta_j)\,d\eta
\end{align}
where
\begin{align}
\tilde{S}^{(m)}(\eta,X)&=
\prod_{i<j}\,\sinh^{1-m\,(n_i+n_j)/2}(a_i-a_j)
\,\prod_{i<j}\,\sinh^{1-m}(\eta_i-\eta_j)\label{Sprime}
\\ \qquad&\cdot\prod_{p=1}^r\,\left[
\prod_{i=1}^{p-1}\,\sinh(\eta_i-a_p) \,\prod_{i=p}^{r-1}\,\sinh(a_p-\eta_i)
\right]^{m\,n_p/2-1}\nonumber
\end{align}
\end{theorem}

Note that the definition of $E(X)$ remains the same but that $\dim E(X)=r-1$.

\begin{proof}We have $\gamma_p=\frac{\prod_{i=1}^{r-1}(e^{2\,\eta_i}-e^{2\,a_p})}
{\prod_{i\not=p}(e^{2\,a_i}-e^{2\,a_p})}$, $p=1$, \dots, $r$.  Now,
\begin{align}
\lefteqn{\frac{\partial(\gamma_1,\dots,\gamma_{r-1})}{\partial(\eta_1,\dots,\eta_{r-1})}
=\left|\det\left[
\frac{2\,e^{2\,\eta_q}\,\prod_{i\not=q}\,(e^{2\,\eta_i}-e^{2\,a_p})}{\prod_{i\not=p}\,(e^{2\,a_i}-e^{2\,a_p})}
\right]_{1\leq p,q\leq r-1}\right|}\nonumber\\
&=2^{r-1}
\,\frac{\prod_{p=1}^{r-1}\,\prod_{i=1}^{r-1}\,(e^{2\,\eta_i}-e^{2\,a_p})}{\prod_{p=1}^{r-1}\,\prod_{i\not=p}\,(e^{2\,a_i}-e^{2\,a_p})}
\,\left|\det\left[\frac{1}{1-e^{2\,a_p}\,e^{-2\,\eta_q}}\right]_{1\leq p,q\leq r-1}\right|
\nonumber\\
&=2^{r-1}
\,\frac{\prod_{i<j<r}\,(e^{2\,\eta_i}-e^{2\,\eta_j})}{\prod_{i<j\leq r}\,(e^{2\,a_i}-e^{2\,a_j})}
\label{J}\\&
=e^{r\,\sum_{i=1}^{r-1}\,\eta_i}\,e^{-(r-1)\,\sum_{i=1}^{r}\,a_i}
\,\frac{\prod_{i<j<r}\sinh(\eta_i-\eta_j)}{\prod_{i<j\leq r}\,\sinh(a_i-a_j)}\nonumber
\end{align}
using \cite[page 202]{Weyl}
\begin{align*}
\det\left [\frac{1}{1-\nu _i\lambda _k}\right ] &=\frac{\prod _{i>k}(\nu
_i-\nu _k)\prod_{i>k} (\lambda _i-\lambda _k)}{\prod_{i=1}^n\prod_{k=1}^n (1
-\nu _i\lambda _k) }.
\end{align*}

On the other hand,
\begin{align}
\prod_{p=1}^r\,\gamma_i^{m\,n_p/2-1}
&=e^{(m\,n/2-r)\,\sum_{k=1}^{r-1}\,\eta_k+\sum_{k=1}^r\,(m\,n_k/2+r-1-m\,n/2)\,a_k}\label{G}\\&\qquad
\cdot\frac{\prod_{p=1}^r\,
\left[
\prod_{i=1}^{p-1}\,\sinh(\eta_i-a_p)
\,\prod_{i=p}^{r-1}\,\sinh(a_p-\eta_i)
\right]^{m\,n_p/2-1}}
{\prod_{i<j}\,\sinh^{m\,(n_i+n_j)/2-2}(a_i-a_j)}.\nonumber
\end{align}

Recall that $\phi_\lambda^{(m)}(e^X)=\chi_{\lambda+i\,\rho^{(m)}_n}^{(m)}(e^X)$ and that $\rho^{(m)}(X)=\frac{m}{2}\,\sum_{k=1}^n\,(n+1-2\,k)\,x_k$.
Therefore, 
\begin{align*}
\phi_\lambda^{(m)}(e^X)&=\frac{\prod_{i=1}^r\,\Gamma(m\,n_i/2)}{\Gamma(m/2)^n}
\,e^{i\,\lambda_n\,\sum_{k=1}^r\,n_k\,a_k}
\,\prod_{i<j}\,\sinh^{1-m\,(n_i+n_j)/2}(a_i-a_j)	
\\ &\cdot
\,\int_{E(X)}\,\phi^{(m)}_{\lambda_0}(e^\xi)
\,\prod_{p=1}^r\,
\left[
\prod_{i=1}^{p-1}\,\sinh(\eta_i-a_p)
\,\prod_{i=p}^{r-1}\,\sinh(a_p-\eta_i)
\right]^{m\,n_p/2-1}
\\&\cdot
\prod_{i<j}\,\sinh(\eta_i-\eta_j)\,d\eta.
\end{align*}

We have used (\ref{J}), (\ref{G}), Remark \ref{trace} and the fact that $\sum_{k=1}^{n-1}\,\xi_k=\sum_{k=1}^r\,(n_k-1)\,a_k+\sum_{k=1}^{r-1}\,\eta_k$.
\end{proof}

\section{Representation of the spherical function}\label{Rep}

We will now derive a Laplace type representation for the spherical functions namely 
\begin{align*}
\phi^{(m)}_\lambda(e^X)=\int_{\a}\,e^{i\,\langle \lambda,H\rangle}\,K^{(m)}_n(H,X)\,dH
\end{align*}
and, at the same time, specify the support of the function $K^{(m)}_n(\cdot,X)$.

Recall that $C(X)$, the support of the spherical functions in the geometric case, is the convex envelop of $W\cdot X$, where $W$ is the Weyl group.   
In this section, we will show that $C(X)$ remains the support in the case of the generalized spherical functions (specifically when $m>0$).
Given the recursive formula for $\phi^{(m)}_\lambda$, it is not surprising that our first step is to derive an inductive description of $C(X)$.

\begin{definition}\label{A0}
Fix $X=\diag[x_1,\dots,x_n]\in\overline{\a^+}$.  For $H\in C(X)$, let $H'=\diag[h_1,\dots,h_{n-1}]$.  For $Y\in C_{n-1}(\xi)$, 
$\xi\in E(X)$, let $\hat{Y}=\diag[y_1,\dots,y_{n-1},\sum_{k=1}^n\,x_k-\sum_{k=1}^{n-1}\,\xi_k]$ and 
$\widehat{C(\xi)}=\{\hat{Y}\colon Y\in C(\xi)\}$.  
\end{definition}

The following characterization of $C(X)$ and $C(X)^\circ$ will prove useful.

\begin{remark}[\cite{Rado}]\label{desc}
If $X\in\overline{\a^+}$ then $H\in\a$ belongs to $C(X)$ if and only if $\tr X=\tr H$ and
\begin{align}
h_{i_1}+\dots+h_{i_k}&\leq x_1+\dots+x_k\label{ineq}
\end{align}
for every choice of distinct indices $i_1$, \dots, $i_k\in\{1,\dots,n\}$,   $1\leq k\leq n-1$.
Moreover, $H\in C(X)^\circ$ if and only if all the inequalities in (\ref{ineq}) are strict.
\end{remark}

\begin{proposition}\label{CONV}
If $X\in \overline{\a^+}$ then $C_n(X)=\cup_{\xi\in E(X)}\,\widehat{C_{n-1}(\xi)}$.
\end{proposition}

\begin{proof}

We first show that $A(X)=\cup_{\xi\in E(X)}\,\widehat{C_{n-1}(\xi)}$ is a convex set.  
Let $H$ and $\tilde{H}\in A(X)$.  We have $H'=\sum_{\sigma \in S_{n-1}}\,\sigma\cdot \xi$ and 
$\tilde{H}'=\sum_{\sigma \in S_{n-1}}\,\sigma\cdot \tilde{\xi}$.  For for $t\in[0,1]$, we have
\begin{align*}
t\,H'+(1-t)\,\tilde{H}'
&=t\,\sum_{\sigma \in S_{n-1}}\,\sigma\cdot \xi+(1-t)\,\sum_{\sigma \in S_{n-1}}\,\sigma\cdot \tilde{\xi}\\
&=\sum_{\sigma \in S_{n-1}}\,\sigma\cdot (t\,\xi+(1-t)\, \tilde{\xi}).
\end{align*} 
Since $E(X)$ is convex, we have $\eta=t\,\xi+(1-t)\, \tilde{\xi}\in E(X)$.  Finally, it follows easily that 
$t\,H+(1-t)\,\tilde{H}\in \widehat{C(\eta)}$.

We next show that $C(X)\subset A(X)$.  Given that $A(X)$ is convex, we only need to show that $\sigma \cdot X\in A(X)$ for every $\sigma\in S_n$.  
For such a $\sigma$, we have $\sigma \cdot X=\diag[x_{\sigma^{-1}},\dots,x_{\sigma^{-1}(n)}]$.  Now define $\tau\in S_{n-1}$ by
$\tau^{-1}(k)=\left\lbrace\begin{array}{cc}\sigma^{-1}(k)&\hbox{if $\sigma^{-1}(k)<\sigma^{-1}(n)$}\\
\sigma^{-1}(k)-1&\hbox{if $k\geq \sigma^{-1}(n)$}\end{array}\right.$ for $k=1$, \dots, $n-1$ and $\xi$ by 
$\xi_k=\left\lbrace\begin{array}{cc}x_k&\hbox{if $k<\sigma^{-1}(n)$}\\
x_{k+1}&\hbox{if $\sigma^{-1}(k)\geq \sigma^{-1}(n)$}\end{array}\right.$ for $k=1$, \dots, $n-1$.  One verifies that $\xi\in E(H)$ and 
that $\sigma\cdot X=\widehat{\tau\cdot \xi}\in \widehat{C_{n-1}(\xi)}\subset A(X)$.

We now claim that $A(X)\subset C(X)$. 
Let $H\in \widehat{C(\xi)}$ for some $\xi \in E(X)$.  Since $H=\sum_{\tau\in S_{n-1}}\,b_\tau\,
\widehat{\tau\cdot \xi}$, it suffices to observe that $\widehat{\tau\cdot \xi}
=\tilde{\tau}\cdot \hat{\xi}$ where $\tilde{\tau}\in 
S_n$ is defined by $\tilde{\tau}(k)=\tau(k)$ if $k<n$ and $\tilde{\tau}(n)=n$.   Since $\hat{\xi}\in C(X)$ (which can be checked using Remark \ref{desc}), the claim follows.
\end{proof}

The next result sets the stage for the proof that the kernel $H\mapsto K(H,X)$ is strictly positive on $C(X)^\circ$.

\begin{proposition}\label{CONV2}
Suppose $X\in \overline{\a^+}$, $X\not=C\,I_n$, and use the notation in Theorem \ref{walls}. 
Let $E(X)^\circ=\{\xi\in E(X)\colon a_{k+1}<\eta_k<a_k\}$ where $\eta_k=\xi_{N_k}$, 
$1\leq k\leq r-1$.  Then $C(X)^\circ=\cup_{\xi\in E(X)^\circ}\,\widehat{C_{n-1}(\xi)^\circ}$
(if $n=2$ then assume $C_{n-1}(\xi)=C_{n-1}(\xi)^\circ=\{\xi\}$).
\end{proposition}

\begin{proof}
Let $H\in C(X)^\circ$.  We know by Proposition \ref{CONV} that there exists $\xi\in E(X)$ with $H'\in C(\xi)$.  We 
prove first that $\xi$ can be chosen so that $H'\in C(\xi)^\circ$.  From Remark \ref{desc}, 
for every distinct $i_1$, \dots, $i_{n-1}\in\{1,\dots,n-1\}$: $h_1+\dots+h_{n-1}=\xi_1+\dots+\xi_{n-1}$ and 
\begin{align}
h_{i_1}+\dots+h_{i_k}&\leq \xi_1+\dots+\xi_k\label{xi},~~1\leq k<n-1.
\end{align}

Suppose $H'\not\in C(\xi)^\circ$ and assume $h_{i_1}\leq\dots\leq h_{i_{n-1}}$.  
Let $k<n-1$ be the smallest index for which an inequality in (\ref{xi}) is not strict.  
We must have $\xi_i<x_i$ for some $i\leq k$ (otherwise $h_{i_1}+\dots +h_{i_k}=\xi_1+\dots+\xi_k=x_1+\dots+x_k$ which contradicts $H\in C(X)^\circ$).
We must also have $\xi_j>x_{j+1}$ for some $j> k$; otherwise $h_{i_1}+\dots+h_{i_k}= 
\xi_1+\dots+\xi_k$ and $h_{i_1}+\dots+h_{i_{n-1}}= \xi_1+\dots+\xi_{n-1}$ would mean 
$h_{i_{k+1}}+\dots+h_{i_{n-1}}= x_{k+2}+\dots+x_n$ \emph{i.e.} 
$\sum_{{\genfrac{}{}{0pt}{}{r\not=i_s,}{s\geq k+1}}}\,h_r
= x_1+\dots+x_{k+1}$ which contradicts $H\in C(X)^\circ$.   Let $\tilde{\xi}_i=\xi_i+\delta$, 
$\tilde{\xi_j}=\xi_j-\delta$ and $\tilde{\xi}_s=\xi_s$ if $s\not=i$ and $s\not=j$ where
$0<\delta<\min\{x_i-\xi_i,\xi_j-x_{j+1}\}$.  
We replace $\xi$ by $\tilde{\xi}$ and note that $H'\in C(\tilde{\xi})$.  If $H'\not\in  C(\tilde{\xi})^\circ$ then the smallest $k$ 
for which an inequality in (\ref{xi}) is not strict will be larger.  Eventually, the process will stop.

We may now assume that $H'\in C(\xi)^\circ$.  We now show that we can select $\xi$ with $\eta_i<a_i$ for $i=1$, \dots, $r-1$.
Suppose $\xi$ does not satisfy that condition.  Let $k\leq r-1$ be the smallest index such that $\eta_k=a_k$.  
If $k=1$, observe that there will be an index $j>1$ such that $\eta_j<a_j$ 
(otherwise $h_{i_1}+\dots+h_{i_{n-1}}=x_1+\dots+x_{n-1}$ which contradicts $H\in C(X)^\circ$); let $\tilde{\eta}_1=\eta_1-
\delta$, $\tilde{\eta}_j=\eta_j+\delta$ and $\tilde{\eta}_i=\eta_i$ if $i\not=1$ and $i\not=j$ with $0<\delta<\min\{\eta_1-a_2=a_1-a_2,a_j-\eta_j,\xi_1+
\dots+\xi_s-h_{i_1}-\dots-h_{i_s},1\leq s\leq j\}$. We replace $\xi$ by $\tilde{\xi}$  by changing $\eta$ and $\tilde{\eta}$. 
We may now assume that $k>1$; let $\tilde{\eta}_1=\eta_1+\delta$, $\tilde{\eta}_k=\eta_k-\delta$ and $\tilde{\eta}_i=\eta_i$ for $i\not=1$ and $i\not=k$ with $0<\delta<
\min\{a_1-\eta_1,a_k-a_{k+1}\}$.  We still have $H'\in C(\tilde{\xi})^\circ$ and the smallest index such that $\tilde{\eta}_i=a_i$, if any, will be larger. 
Eventually, the process will stop.

We now assume that $H\in C(\xi)^\circ$ and that $\xi$ satisfies $\eta_i<a_i$ for each $i$.  To ensure that $\eta_i>a_{i+1}$ for all $i$, we proceed in much the same way except that we consider the largest index such that $\eta_k=a_{k+1}$ (if any).  The rest is as before.

Now let $H\in \widehat{C(\xi)^\circ}$ with $\xi\in E(X)^\circ$ and suppose $H\not\in C(X)^\circ$.  Since 
$H'\in C(\xi)^\circ$, we must have $h_{i_1}+\dots+h_{i_{s-1}}+h_n=x_1+\dots+x_s$ for some $s<n$. 
Therefore, since $H'\in C(\xi)^\circ$ and $\xi\in E(X)^\circ$, $\xi_1+\dots
+\xi_{n-1}=h_1+\dots+h_{n-1}=h_1+\dots+h_{n-1}+h_n-h_n=x_1+\dots+x_n-h_n=(h_{i_1}+\dots+h_{i_{s-1}}+h_n)+(x_{s+1}+\dots+x_n)-h_n=(h_{i_1}+\dots+h_{i_{s-1}})+
(x_{s+1}+\dots+x_n)<(\xi_1+\dots+\xi_{s-1})+(\xi_s+\dots+\xi_{n-1})$ which is absurd.
\end{proof}

\begin{remark}\label{dense}
Suppose $H\in C(X)^\circ$ and $H'\in C(\xi)$ with $\xi\in E(H)$.  By choosing the successive $\delta$'s in the proof small enough, one can choose 
$\tilde{\xi}\in E(X)^\circ$ to be arbitrarily close to $\xi$ with $H'\in C(\tilde{\xi})^\circ$.
\end{remark} 

We can now provide a Laplace-type representation for the generalized spherical function associated to the root system $A_{n-1}$ along with the support of the dual of the Abel transform.

\begin{theorem}\label{CX}
Assume $m>0$ and suppose $\phi_\lambda^{(m)}$ is the generalized spherical function for the root system $A_{n-1}$ in the trigonometric setting 
and, for $X\in\a$, $X\not=c\,I_n$ and $H\in C(X)$, let $D_H(X)=\{\xi\in E(X)\colon H'\in C(\xi)\}$.  Then 
\begin{align}
\phi_\lambda^{(m)}(e^X)&=\int_{C(X)}\,e^{i\,\lambda(H)}\,K^{(m)}_n(H,X)\,dH
\label{pos}
\end{align}
where 
\begin{align}
K^{(m)}_n(H,X)&=
\frac{\Gamma(m\,n/2)}{\prod_{i=1}^r\,\Gamma(m\,n_i/2)}
\,\int_{D_H(X)} \,K^{(m)}_{n-1}(H',\xi)\,\tilde{S}^{(m)}(\eta,X)\,\prod_{i<j}\,\sinh^m(\eta_i-\eta_j)\,d\eta_1\cdots d\eta_{r-2}\label{Km}
\end{align}
is strictly positive and smooth on $C(X)^\circ$.  When $r=2$, $D_H(X)$ contains only one element  and (\ref{Km}) should be interpreted as 
\begin{align*}
K^{(m)}_n(H,X)&=
\frac{\Gamma(m)}{\prod_{i=1}^2\,\Gamma(m\,n_i/2)}
K^{(m)}_{n-1}(H',\xi)\,\tilde{S}^{(m)}(\eta,X).
\end{align*}
Furthermore, if $n=2$ then $K^{(m)}_{n-1}(H',\xi)$ should be replaced by 1.
\end{theorem}

\begin{remark}\label{Rm}
If we consider the larger range $\Re m>0$ then the result remains valid except that $K^{(m)}_n(\cdot,X)\not=0$ on $C(X)^\circ$ does not necessarily follow.  
All we can conclude then is that $\hbox{support}\,K^{(m)}_n(\cdot,X)\subset C(X)$.

In the geometric case, when $X=c\,I_n$, the measure in (\ref{pos}) is the Dirac measure $\delta_X$.   Using Proposition \ref{psi2} and induction, we can conclude that the same holds in the generalized setting.
\end{remark}

\begin{proof}
Assume $X\in\a$, $X\not=c\,I_n$. Since $\phi_\lambda^{(m)}(e^X)=\phi_\lambda^{(m)}(e^{\pi(X)})$, we can assume without loss of generality that $X\in\overline{\a^+}$.
We prove the result by induction on $n\geq 2$.  If $n=2$ then we have
\begin{align}
K^{(m)}_2(H,X)&=\frac{\Gamma(m)}{(\Gamma(m/2))^2}\,\sinh^{1-m}(x_1-x_2)\,[\sinh(x_1-h_1)\,\sinh(h_2-x_2)]^{m/2-1}
\label{K2}
\end{align}
which is smooth and strictly positive when $x_1>h_1\geq h_2>x_2$ \emph{i.e.} when $H\in C(X)^\circ$ ($X\in\overline{\a^+}$, $X\not=C\,I_2$ implies $X\in\a^+$).  Note that if
$m\not=2$, $K^{(m)}_2(H,X)$ is either equal to 0 (when $m>2$) or infinite (when $0<m<2$) when $H\in\partial C(X)$; $K^{(2)}_2(H,X)>0$ for 
$H\in C(X)$ and is 0 elsewhere.  

Assume now that the result is true for $n-1$, $n\geq 3$.  Suppose first that $r>2$.  For $H\in C(X)^\circ$, let 
\begin{align}
D_H(X)&=\{\xi\in E(X)\colon H'\in C(\xi)\},\nonumber\\
D^\circ_H(X)&=\{\xi\in E(X)^\circ\colon H'\in C(\xi)^\circ\},\label{DH}\\
E_H(X)&=\{\xi\in E(X)\colon \tr \xi=\tr H'\}.\nonumber
\end{align}

The sets in (\ref{DH}) can all be parametrized by $\eta_1$, \dots, $\eta_{r-2}$ since
$h_1+\dots+h_{n-1}=\xi_1+\dots+\xi_{n-1}=\eta_1+\dots+\eta_{r-1}+\sum_{i=1}^r\,(n_i-1)\,a_i=h_1+\dots+h_{n-1}$.  
Furthermore, $D^\circ_H(X)\subseteq D_H(X)\subseteq E_H(X)$.  Both $D_H(X)$ and $E_H(X)$ are closed sets.
We claim that that $D^\circ_H(X)$ is a nonempty open subset of $E_H(X)$ which is dense in $D_H(X)$. 
Indeed, observe first that $D^\circ_H(X)\not=\emptyset$ is a consequence of Proposition \ref{CONV2}.  
Now, let $\tilde{\xi}\in D^\circ_H(X)$ and let $0<\epsilon<
\min\{
(\xi_1+\dots+\xi_s-h_{i_1}-\dots-h_{i_s})/s, s=1,\dots,n-2, 
a_i-\eta_i,\eta_i-a_{i+1}
\}$ (the indices $i_s$ are assumed to be distinct and between 1 and $n-1$).  It is not difficult to check that if $\xi\in E_H(X)$ and 
$|\eta_k-\tilde{\eta}_k|<\epsilon$ for all $k$ then $\xi\in E(X)^\circ$ and $H'\in C(\xi)^\circ$ \emph{i.e.} 
that $\xi\in D^\circ_H(X)$.  The fact that $D^\circ_H(X)$ is dense in $D_H(X)$ follows easily from Remark \ref{dense}.  We have
\begin{align}
\phi^{(m)}_\lambda(e^X)&=
\frac{\Gamma(m\,n/2)}{\prod_{i=1}^r\,\Gamma(m\,n_i/2)}
\,e^{i\,\lambda_n\,\sum_{k=1}^r\,n_k\,a_k}
\label{integrand}
\\\qquad&
\cdot\int_{E(X)}\,\int_{C(\xi)}\,e^{i\,\lambda_0(H')}\,K^{(m)}_{n-1}(H',\xi)\,dH'
\,\tilde{S}^{(m)}(\eta,X)\,\prod_{i<j}\,\sinh^m(\eta_i-\eta_j)\,d\eta\nonumber
\end{align}

From (\ref{integrand}), we obtain expression (\ref{pos}) with
\begin{align*}
K^{(m)}_n(H,X)&=
\frac{\Gamma(m\,n/2)}{\prod_{i=1}^r\,\Gamma(m\,n_i/2)}
\,\int_{D_H(X)} \,K^{(m)}_{n-1}(H',\xi)\,\tilde{S}^{(m)}(\eta,X)\,\prod_{i<j}\,\sinh^m(\eta_i-\eta_j)\,d\eta_1\cdots d\eta_{r-2}\\
&=
\frac{\Gamma(m\,n/2)}{\prod_{i=1}^r\,\Gamma(m\,n_i/2)}
\,\int_{D^\circ_H(X)} \,K^{(m)}_{n-1}(H',\xi)\,\tilde{S}^{(m)}(\eta,X)\,\prod_{i<j}\,\sinh^m(\eta_i-\eta_j)
\,d\eta_1\cdots d\eta_{r-2}>0
\end{align*}
since the integrand is smooth and strictly positive over an open subset of $E_H(X)$.

If $r=2$ then $D_H(X)=\{\xi\}$ with $\eta=\eta_1=\sum_{k=1}^{n-1}h_k-(n_1-1)\,a_1-(n_2-1)\,a_2$ and
\begin{align*}
K^{(m)}_n(H,X)&=\frac{\Gamma(m)}{\prod_{i=1}^2\,\Gamma(m\,n_i/2)}
K^{(m)}_{n-1}(H',\xi)\,\tilde{S}^{(m)}(\eta,X)
\end{align*}
and the result follows.
\end{proof}

\begin{remark}
Graczyk and Loeb have given a fairly explicit construction for the kernel of the Abel transform in the case of 
complex Lie groups in \cite{Graczyk} while Graczyk and Sawyer have also given in \cite{PGPS4} an expression in the 
case of Lie groups of noncompact type (with a few exceptions of low dimension).  In either case, the support of the 
kernel is not obvious from the expression (although known since we are in the geometric setting).  In \cite{Trimeche}, Trim\`eche has provided sensibly the same 
expression as we have here in the case $A_2$ ($n=3$) and showed that the support of the kernel is included in $C(X)$.
\end{remark}

\section{The Dunkl setting}\label{Dunkl}

We use a result of de Jeu in \cite{DeJeu} (taking ``rational limits'') to adapt the results of Section \ref{generalized} and Section \ref{Rep} to the Dunkl setting. 
Using a result from Rejeb, this allows us in turn to describe precisely the support of the Dunkl intertwining operator $V$. 

\begin{theorem}\label{psidunkl}
We use the notation set up in Definition \ref{notaplus}.  
The generalized Dunkl spherical function associated to the root system $A_{n-1}$ in the Dunkl settingis given by 
$\psi_\lambda^{(m)}(e^X)=e^{i\,\lambda(X)}$ when $n=1$ and, for 
$n\geq 2$,
\begin{align}
\psi_\lambda^{(m)}(e^X)
&=\frac{\Gamma(m\,n/2)}{\prod_{i=1}^r\,\Gamma(m\,n_i/2)}
\,e^{i\,\lambda_n\,\sum_{k=1}^r\,n_k\,a_k}
\,\int_{E(X)}\,\psi^{(m)}_{\lambda_0}(e^\xi)\,T^{(m)}(\eta,X)
\,d_0(\eta)^m\,d\eta\label{Spher0}
\end{align}
where $X\in\overline{\a^+}$, $E(X)$, and $\lambda_0$ are as before,  $d_0(\eta)=\prod_{r<s}\,(\eta_r-\eta_s)$ and
\begin{align*}
T^{(m)}(\eta,X)&=
\prod_{i<j}\,(a_i-a_j)^{1-m\,(n_i+n_j)/2} \,d_0(\xi)^{1-m}
\prod_{p=1}^r\,\left[
\prod_{i=1}^{p-1}\,(\eta_i-a_p) \,\prod_{i=p}^{r-1}\,(a_p-\eta_i)
\right]^{m\,n_p/2-1}.
\end{align*}
\end{theorem}

\begin{proof}
We apply the Weyl-invariant version of \cite[Theorem 4.13]{DeJeu}: we have
\begin{align*}
\psi^{(m)}_\lambda(e^X)&=\lim_{\epsilon\to0}\,\phi^{(m)}_{\lambda/\epsilon}(e^{\epsilon\,X})
\end{align*}
uniformly in $\lambda$ and $X$ over compact sets (this technique has been described as ``taking rational limits'').  We then use the change of variable $\eta=\epsilon\,\tilde{\eta}$ which means that $\xi=\epsilon\,\tilde{\xi}$ and 
$d\eta=\epsilon^{r-1}\,d\tilde{\eta}$.  Counting the powers of $\epsilon$ carefully, we have
\begin{align}
\lefteqn{\phi^{(m)}_{\epsilon\,\lambda_0}(e^\xi)\,\tilde{S}^{(m)}(\eta,\epsilon\,X)
\,\prod_{i<j}\,\sinh^m(\eta_i-\eta_j)\,d\eta}\label{mess}\\
&=
\phi^{(m)}_{\epsilon\,\lambda_0}(e^{\epsilon\,\tilde{\xi}})\,\prod_{i<j}\,\left[\frac{\sinh(\epsilon\,a_i-\epsilon\,a_j)}{\epsilon}\right]^{1-m\,(n_i+n_j)/2}
\,\prod_{i<j}\,\frac{\sinh(\epsilon\,\tilde{\eta}_i-\epsilon\,\tilde{\eta}_j)}{\epsilon} \nonumber
\\ \qquad&\cdot\prod_{p=1}^r\,\left[
\prod_{i=1}^{p-1}\,\frac{\sinh(\epsilon\,\tilde{\eta}_i-\epsilon\,a_p)}{\epsilon} \,\prod_{i=p}^{r-1}\,\frac{\sinh(\epsilon\,a_p-\epsilon\,\tilde{\eta}_i)}{\epsilon}
\right]^{m\,n_p/2-1}\,d\tilde{\eta}.\nonumber
\end{align} 

The term 
\begin{align*}
\lefteqn{\prod_{i<j}\,\left[\frac{\sinh(\epsilon\,a_i-\epsilon\,a_j)}{\epsilon}\right]^{1-m\,(n_i+n_j)/2}
\,\prod_{i<j}\,\frac{\sinh(\epsilon\,\tilde{\eta}_i-\epsilon\,\tilde{\eta}_j) }{\epsilon}
}
\\ \qquad&\cdot\prod_{p=1}^r\,\left[
\prod_{i=1}^{p-1}\,\frac{\sinh(\epsilon\,\tilde{\eta}_i-\epsilon\,a_p)}{\epsilon} \,\prod_{i=p}^{r-1}\,\frac{\sinh(\epsilon\,a_p-\epsilon\,\tilde{\eta}_i)}{\epsilon}
	\right]^{m\,n_p/2-1}
\end{align*}
converges to $T^{(m)}(\eta,X)\,d_0(\eta)^m$ and is dominated by $M_X\,T^{(m)}(\eta,X)\,d_0(\eta)^m$, where $M_X$ is a constant which depends only on the compact set $E(X)$ and therefore on $X$
(this last statement is easily inferred from the inequalities $\frac{1}{2}\,e^u\,\frac{u}{1+u}\leq\sinh u\leq e^u\,\frac{u}{1+u}$ for $u\geq 0$).

Another application of \cite[Theorem 4.13]{DeJeu} gives us
$\lim_{\epsilon\to0}\,\phi^{(m)}_{\epsilon\,\lambda_0/\epsilon}(e^{\epsilon\tilde{\xi}})=\psi_{\lambda_0}(e^{\tilde{\xi}})$ uniformly on $E(H)$; 
the result then follows using the dominated convergence theorem.
\end{proof}

\begin{corollary}
We use the same definitions as in the theorem.  Then
\begin{align}
\psi_\lambda^{(m)}(e^X)
&=\frac{\Gamma(m\,n/2)}{\prod_{i=1}^r\,\Gamma(m\,n_i/2)}
\,e^{i\,\lambda_n\,\sum_{k=1}^r\,n_k\,a_k}
\,\int_{\sigma_r}\,\psi^{(m)}_{\lambda_0}(e^{\xi})
\,\prod_{i=1}^n\gamma_i^{m\,n_i/2-1}\,d\gamma\label{Spher1}
\end{align}
where the $\xi_i$ are the roots of $q(x)=\prod_{j=1}^r\,(x-a_j)^{n_j-1}\,\sum_{i=1}^r
\,\gamma_i\,\prod_{j\not=i}\,(x-a_j)$.

More precisely, $\xi$ is as in (\ref{aa}) with $\xi_k=a_i$ whenever $N_{i-1}<k<N_i$ and $\xi_{N_i}=\eta_i
\in [a_i,a_{i+1}]$ are the roots of the polynomial $q_1(x)=\sum_{i=1}^r\,\gamma_i\,\prod_{j\not=i}\,(x-a_j)$.  
Equivalently,
\begin{align*}
\gamma_p&=\frac{\prod_{j=1}^{r-1}\,(\eta_j-a_p)}{\prod_{j\not=p}\,(a_j-a_p)},
~p=1,\dots,r.
\end{align*}
\end{corollary}

\begin{proof}
It suffices to use the change of variables 
$\gamma_p=\frac{\prod_{j=1}^{r-1}\,(\eta_j-a_p)}{\prod_{j\not=p}\,(a_j-a_p)}$,  $p=1,\dots,r$ in (\ref{Spher0}).
\end{proof}

\begin{proposition}\label{Veh}
Let $f\mapsto Vf(X)=\int_{\a}\,f(H)\,d\mu_X(H)$ be the intertwining operator in the Dunkl setting as discussed in the Introduction and let $f\mapsto\dual{}f$ the dual Abel operator.
\begin{enumerate}
\item If $f$ is a Weyl-invariant smooth function then $Vf=\dual{}f$.
\item $\hbox{support}(\mu_X)$ is Weyl-invariant.
\item $\hbox{support}(K(\cdot,X))=\hbox{support}(\mu_X)\subset C(X)$
\end{enumerate}
\end{proposition}

\begin{proof}
\begin{enumerate}
\item Suppose that $f$ is Weyl-invariant.  Then for $X\in\a$,
\begin{align*}
\dual{}f(X)
&=\dual\left(\int_{\a^*}\,\tilde{f}(\lambda)\,e^{-i\,\lambda(X)}\,d\lambda\right)
=\int_{\a^*}\,\tilde{f}(\lambda)\,\dual(e^{-i\,\lambda(X)})\,d\lambda\\
&=\int_{\a^*}\,\tilde{f}(\lambda)\,\psi_{-\lambda}(X)\,d\lambda
\end{align*}
while
\begin{align*}
Vf(X)
&=V\left(\int_{\a^*}\,\tilde{f}(\lambda)\,e^{-i\,\lambda(X)}\,d\lambda\right)
=V\left(\int_{\a^*}\,\tilde{f}(\lambda)\,\frac{1}{|W|}\,\sum_{\sigma\in W}\,e^{-i\,\lambda(\sigma^{-1}\cdot X)}\,d\lambda\right)\\
&=\int_{\a^*}\,\tilde{f}(\lambda)\,V\left(\frac{1}{|W|}\,\sum_{\sigma\in W}\,e^{-i\,\lambda(\sigma^{-1}\cdot X)}
\right)\,d\lambda
=\int_{\a^*}\,\tilde{f}(\lambda)\,\psi_{-\lambda}(X)\,d\lambda
\end{align*}
the last equality being a consequence of the relationship between the generalized spherical functions and the 
eigenfunctions of the Dunkl operators as given in (\ref{FG}) (refer to \cite{Opdam1} for example).
\item This result appears in the doctoral thesis of C.{} Rejeb \cite[Theorem 2.9]{Rejeb}.
\item $\hbox{support}(K(\cdot,X))\subset \hbox{support}(\mu_X)$.  Indeed, suppose that $H_0\not\in\hbox{support}(\mu_X)$.  Using 2.{}, this implies that $(W\cdot H_0)\cap\hbox{support}(
\mu_X)=\emptyset$. Let $\epsilon>0$ be such that $(W\cdot B(H_0,2\,\epsilon))\cap\hbox{support}(\mu_X)=\emptyset$ and let $f$ be a smooth non-negative function which is 
identically 1 on $\overline{B(H_0,\epsilon)}$ and 0 outside $B(H_0,2\,\epsilon)$. Let $f^W(Z)=\frac{1}{|W|}\,\sum_{
\sigma\in W}\,f(\sigma^{-1}\cdot Z)$.  Then for $X\in\a$,
\begin{align*}
0&=V(f^W)(X)=\dual{}f^W(X)=\dual{}f(X)
\end{align*}
which means that $H_0\not\in \hbox{support}(K(\cdot,X))$.
\end{enumerate}

On the other hand, if $\mu_X(U)>0$ with $U$ open, then
\begin{align*}
\dual(\un_U)(X)&=\dual\left(\frac{\sum_{w\in W}\, \un_U\circ w^{-1}}{|W|}\right)(X)
              =V    \left(\frac{\sum_{w\in W}\, \un_U\circ w^{-1}}{|W|}\right)(X)
							\\&
\geq \frac{V(\un_U)(X)}{|W|}>0
\end{align*}
which implies that $\hbox{support}(\mu_X)\subset \hbox{support}(K(\cdot,X))$.
\end{proof}

\begin{remark}
R\"osler and Voit have shown in \cite{Roesler0, Roesler2} that the support of the intertwining operator $V$ is always included in $C(X)$.  Furthermore, our result shows that the support of the Abel transform and of the intertwining operator are the same.
\end{remark}

We show below that the results for the Laplace-type representation of the generalized spherical functions associated to the root systems of type $A$ still hold in the Dunkl setting.  Furthermore,
we are able to describe precisely the support of the Dunkl intertwining operator $V$.

\begin{theorem}\label{CXfinal}
Assume $m>0$ and suppose $\psi_\lambda^{(m)}$ is the generalized spherical function for the root system $A_{n-1}$ in the Dunkl setting 
and, for $X\in\a$, $X\not=c\,I_n$ and $H\in C(X)$, let $D_H(X)=\{\xi\in E(X)\colon H'\in C(\xi)\}$.  Then 
\begin{align*}
\psi_\lambda^{(m)}(e^X)&=\int_{C(X)}\,e^{i\,\lambda(H)}\,\check{K}^{(m)}_n(H,X)\,dH
\end{align*}
where 
\begin{align}
\check{K}^{(m)}_n(H,X)&=
\frac{\Gamma(m\,n/2)}{\prod_{i=1}^r\,\Gamma(m\,n_i/2)}
\,\int_{D_H(X)} \,\check{K}^{(m)}_{n-1}(H',\xi)\,T^{(m)}(\eta,X)\label{Km2}
\\\nonumber&\qquad\cdot\,\prod_{i<j}\,(\eta_i-\eta_j)^m\,d\eta_1\cdots d\eta_{r-2}
\end{align}
is strictly positive and smooth on $C(X)^\circ$.  When $r=2$, $D_H(X)$ contains only one element  and (\ref{Km}) should be interpreted as 
\begin{align*}
\check{K}^{(m)}_n(H,X)&=
\frac{\Gamma(m)}{\prod_{i=1}^2\,\Gamma(m\,n_i/2)}
\check{K}^{(m)}_{n-1}(H',\xi)\,T^{(m)}(\eta,X).
\end{align*}
Furthermore, if $n=2$ then $K^{(m)}_{n-1}(H',\xi)$ should be replaced by 1.
\end{theorem}

Remark \ref{Rm} also applies here.

\begin{proof}
The proof is practically identical to the one of Theorem \ref{CX}.
\end{proof}

Furthermore, we are able to specify precisely the support of the Dunkl intertwining operator $V$.
\begin{corollary}
The support of the intertwining operator $V$ is $C(X)$.
\end{corollary}

\begin{proof}
A consequence of the theorem and of Proposition \ref{Veh}, part 3.
\end{proof}

\section{Conclusion}

Our approach is heavily dependent on the iterative formulae of the spherical functions for the roots system $A_{n-1}$. 
It is not out of question that iterative formulae such as (\ref{Spherical0}), (\ref{psi}), (\ref{Spher0}) and (\ref{Spher1}) could be developed for the classical real and complex symmetric spaces.  These formulae could then potentially lead to  expressions for the generalized spherical functions associated to the corresponding root systems.
Not only was this approach used to derive the original formulae (\ref{Spherical0}) and (\ref{psi}) but the same principle was used to prove a sharp criterion for the existence of the product formula for different root systems 
(see for instance \cite{PGPS1,PGPS2,PGPS3}).

\end{document}